\documentclass{amsart}

\usepackage{latexsym,amsmath,amsfonts,amscd,amssymb,graphicx, xcolor, amsthm}
\usepackage{theoremref}
\usepackage{url}
\usepackage{mathtools}

\theoremstyle{plain}

\newtheorem{theorem}{Theorem}
 \newtheorem{corollary}[theorem]{Corollary} 
  
 \newtheorem{proposition}[theorem]{Proposition}
\theoremstyle{definition}
\newtheorem{definition}[theorem]{Definition}
\newtheorem{example}[theorem]{Example}

\newtheorem{remark}[theorem]{Remark}

\begin{document}

\title{Integer Sequences and Output Arrays}

\author{John P. D'Angelo}

\address{ Dept. of Mathematics, Univ. of Illinois, 
1409 W. Green St., Urbana IL 61801, USA}
\email{ jpda@illinois.edu}

\author{Ji\v{r}\'{\i} Lebl}
\thanks{The second author was in part supported by Simons Foundation collaboration grant 710294.}
\address{Department of Mathematics, Oklahoma State University,
Stillwater, OK 74078, USA}
\email{lebl@okstate.edu}

\maketitle

\begin{abstract} The first author \cite{D1} recently  introduced an integer sequence now numbered A355519 in \cite{OEIS}.
This sequence arose from counting bracket tournaments; its study evokes the analysis of the Catalan triangle 
(sequence A009766 in \cite{OEIS}) and the related Catalan numbers, sequence A000108 in \cite{OEIS}. We therefore introduce a general construction that
places these sequences on the same footing and suggests many new integer sequences. We provide code for performing this construction and a lengthy list of examples.
This construction determines a function from input sequences to output sequences. Some of the resulting output sequences are in \cite{OEIS} and others are not.

\medskip

\noindent {\bf AMS Classification Numbers}:  Primary 11B83, Secondary 11B37.

\medskip

\end{abstract}


\section{Introduction}

The starting point of this paper is an arbitrary non-decreasing sequence of positive integers, called an {\it input sequence}. Given this sequence $(y_n)_{n\ge 1}$, we use it to define an infinite system of inequalities. We use these inequalities to introduce the notion of a valid $n$-tuple. We study the number of valid $n$-tuples via an output array ${\mathcal A} = A(n,k)$ determined by the given sequence.
Here $n$ denotes the row and $k$ denotes the column; $n\ge 1$ and $k\ge 0$.
The entry $A(n,k)$ is the number of valid $n$-tuples with first entry $k$.
In Theorem \ref{recurrence-theorem}  we show that entries in $A(n,k)$ (for appropriate $k$) satisfy the recurrence
\begin{equation} \label{recurrence} A(n+1,k+1) = A(n+1,k) + A(n,k+1). \end{equation}
Thus an entry is the sum of the entry above it and the entry to its left in the array.
Sums of the entries in a row of the output array often define interesting integer sequences, such as binomial coefficients, Catalan numbers, and the number of valid tournament brackets.
We call the sequence of these row sums the {\it output sequence}. The columns of the output array are also interesting;
each column defines a combinatorial polynomial. See Theorem \ref{array-formula}  and Propositions \ref{poly-bracket}, \ref{Cat-polynomial}, and \ref{Fib-polynomial}.

Perhaps the two most interesting examples are the well-known Catalan triangle (see for example \cite{West} and \cite{Stanley})  and the array arising from tournament brackets in \cite{D1}. 
We introduce and study several additional examples, including the array arising from the Fibonacci sequence. It is surprising to the authors that the output sequence in this case
does not appear in \cite{OEIS}. In the final section of the paper we give a lengthy list of input sequences and their corresponding output sequences. 

Let $(y_n)_{n\ge 1}$ be an input sequence, namely a nondecreasing sequence of positive integers. 
For each positive integer $n$ we consider $n$-tuples of non-negative integers satisfying inequalities determined
from this sequence. We make the following definition.

\begin{definition} Let $(y_n)_{n\ge 1}$ be a nondecreasing sequence of positive integers.  For each positive integer $n$, we say that 
${\bf x}= (x_1,\ldots, x_n)$ is a {\bf valid} $n$-tuple for $(y_n)$ if the following hold:
\begin{itemize}
\item $x_1 \le y_n$,
\item for $ j\ge 1 $, we have  $x_{j+1} \le \min( x_j, y_{n-j})$. 
\end{itemize}
\end{definition} 

Notice that the entries in ${\bf x}$ are nonincreasing.  We next define {\bf output arrays}, the main construction in this paper.

\begin{definition}  Fix an input sequence $(y_n)$. For each positive integer $n$ and each non-negative integer $k$, define  $A(n,k)$ to be the number of valid $n$-tuples with $x_1=k$. 
The output array ${\mathcal A}$ is defined by ${\mathcal A}= \left(A(n,k)\right)$. \end{definition}

For the constant sequence with $y_n=j$, an $n$-tuple is valid if we have
\begin{equation} x_n \le x_{n-1} \le \dots \le x_1 \le j. \end{equation}
When $j=1$, there are $j+1$ valid $n$-tuples. When $j=2$, there are 
\begin{equation} 1+2 + \cdots + n = {n+2 \choose 2} \end{equation}
valid $n$-tuples. In fact we have the following well-known result.

\begin{proposition} Fix a positive integer $j$. For each positive integer $n$ there are ${n+j \choose j}$ $n$-tuples ${\bf x}$ of non-negative integers  satisfying
\begin{equation} x_n \le x_{n-1} \le \ldots \le x_1 \le j. \end{equation} 
\end{proposition}

We describe the output arrays for  three natural input sequences. These arrays are easily obtained using the code from Section \ref{sage}.

\begin{example} Consider the constant sequence with $j=5$. Here are the first six rows of ${\mathcal A}$.
\begin{equation} \label{constant5} \begin{pmatrix} 1 & 1 & 1 & 1 & 1 & 1 & 0 & \dots \cr 1 & 2 & 3 & 4 & 5 & 6  & 0 & \cdots \cr 1 & 3 & 6 & 10 & 15 & 21 & 0 & \cdots \cr 1 & 4 & 10 &20 & 35 & 56 & 0 & \cdots \cr 1& 5 & 15 & 35 & 70 & 126 & 0 & \cdots \cr 1 & 6 &21 & 56 & 126 & 252 & 0 & \cdots \end{pmatrix} .\end{equation} 
Note that the row sums are $6,21, 56, 126, 252, 462$. These are of course the values of ${n+5 \choose 5}$ for $1 \le n \le 6$.
\end{example}

\begin{example} Consider the identity sequence. Then the output array ${\mathcal A}$ yields the Catalan triangle:
\begin{equation} \begin{pmatrix} 1 & 1 & 0 & 0 & 0 & 0 & 0 & 0 & \dots \cr 1 & 2 & 2 & 0 & 0 & 0 & 0 & 0  & \cdots \cr 1 & 3 & 5 & 5  & 0 & 0 & 0 & 0 & \cdots \cr 1 & 4 & 9 &14  & 14  & 0 & 0 & 0 & \cdots \cr 1& 5 & 14 & 28 & 42 & 42  & 0 & 0 & \cdots \cr 1 & 6 &20 & 48 & 90 & 132 & 132  & 0 & \cdots \end{pmatrix}. \end{equation}  \end{example}

\begin{example} The third example comes from tournament brackets \cite{D1}. It is more complicated than the previous two examples, because  $A(n,k)  > 0$ for more than $n$ values of $k$.
We list the first four rows and ten columns:
 $$ \begin{pmatrix} 1 & 1 & 0 & 0 & 0 & 0 & 0 & 0 & 0 & 0 \cr 1 & 2 & 2 & 0 & 0 & 0 & 0 & 0 & 0 & 0 \cr 1 & 3 & 5 & 5 & 5 & 0 & 0 & 0 & 0 & 0  \cr 1 & 4 & 9 & 14 & 19 & 19 & 19 & 19 & 19  & 0 \end{pmatrix} .$$
In the fourth row, there must be nine non-zero numbers, because $2^{4-1} +1 = 9$.  Since $2^{3-1}+1 = 5$, there must also be five terms equal to $19$. 

The fifth row is
$$ 1, 5, 14, 28, 47, 66, 85, 104, 123, 123, 123, 123, 123, 123, 123, 123, 123, 0, \ldots$$
We note that the number of non-zero terms grows rapidly, and thus forces the row sums to grow even more rapidly. 
There are nine copies of $123$, because nine numbers $m$ satisfy $8 \le m \le 16$. See \cite{D1} for the sixth row and information on the row sums.
\end{example}

\begin{remark} The phenomenon where  $A(n,k)  > 0$ for more than $n$ values of $k$ applies to all but the simplest input sequences. 
Hence the tournament bracket output array and Fibonacci output array are more typical than the Catalan triangle. \end{remark} 

\begin{definition} Let $(y_n)$ be an input sequence. We define its output sequence $W$ by setting $W(n)$ to be the sum of the entries in the $n$-th row of ${\mathbb A}$. Thus
$$ W(n) = \sum_{k=0}^\infty A(n,k) = \sum_{k=0}^{y_n} A(n,k). $$\end{definition}

The sequence $W(n)$ is also non-decreasing, and in general grows rapidly.

 \begin{definition} Let ${\mathcal S}$ denote the space of non-decreasing integer sequences. Let  $\Phi:{\mathcal S} \to {\mathcal S}$  denote the operation of assigning an output sequence to an input sequence. \end{definition}

Our work in this paper provides considerable information about this operation $\Phi$. For example,  the bounds in Theorem \ref{Pascal-ineq} show that $\Phi$ is far from being surjective. Describing its range precisely seems to be a difficult problem.

\section{Recurrence relation}

The following proposition is immediate from the definitions.

\begin{proposition} The number of valid $n$-tuples is the sum of the entries in the $n$-th row of ${\mathcal A}$. \end{proposition}

By definition, each row of ${\mathcal A}$ ends in an infinite string of zeroes. Since $x_1 \le y_n$ and $x_1$ is the largest of the $x_j$, it follows that $A(n,k) = 0$  for $k > y_n$. For each $n$, it follows also by definition that $A(n,0)=1$. (There is only one valid $n$-tuple whose maximum value equals $0$, namely the tuple with all zeroes.)
Let $K=K_n$ denote the maximum $k$ for which $A(n,k) > 0$. Again, by definition, for each $n$ we have the following inequalities:
\begin{equation} 1 = A(n,0) \le A(n,1) \le \cdots \le A(n,K). \end{equation}

\begin{theorem} \label{recurrence-theorem} Fix $n$. Let $K$ denote the largest $k$ for which 
$A(n+1,k+1)> 0$.
 Suppose that $k+1 \le K$. Then the following recurrences hold:
\begin{equation} \label{recurrence-left} A(n+1, k+1) = A(n+1,k) + A(n,k+1)\end{equation}
\begin{equation} \label{recurrence-above} 
A(n+1,k+1) = \sum_{j=0}^{k+1}  A(n,j). \end{equation}
\end{theorem}
\begin{proof}
Consider a valid $(n+1)$-tuple ${\bf x}$ with $x_1 = k+1$. For appropriate integers we then have ${\bf x} = (k+1, a, b, \ldots, z)$. There are two kinds of such $(n+1)$-tuples:
those for which $a=k+1$ and those for which $a < k+1$. The number of the first type is $A(n,k+1)$. To count the number of the second type, observe (since $a<k+1$) that the number of 
$(n+1)$-tuples of the form $(k+1, a, \ldots)$ is the same as the number of $(n+1)$-tuples of the form $(a,a,\ldots)$. 
Since $a$ is an arbitrary integer 
in the interval $[0, k]$, the number of such tuples is  $A(n+1,k)$.

In other words, each valid $(n+1)$-tuple whose first entry is $k+1$ either has second entry also $k+1$, or has second entry in the integer interval $[0,k]$. 
The number of the first type is $A(n,k+1)$. The second type consists of $(n+1)$-tuples whose first entry is $k+1$ and whose second entry is at most $k$. The number of these is unchanged
if we replace the first entry by $k$. Hence there are $A(n+1,k)$ of the second type.

To prove the second recurrence, observe that a valid $(n+1)$-tuple with $x_1 = k+1$ is of the form $(k+1, a,b, \ldots, z)$. Here the $n$-tuple $(a, \ldots ,z)$ is a valid $n$-tuple as long as
$a \le k+1$. Thus (\ref{recurrence-above}) holds.
\end{proof}

\begin{corollary}\label{pseudo-pascal}  For $k \le y_n$, we find $A(n,k)$ by adding the pair of numbers immediately above it and to its left. We can also find $A(n,k)$ by adding all the numbers in the row above
up to column $k$. \end{corollary}

Corollary \ref{pseudo-pascal} is well known for the Catalan triangle, and was noted for the bracket array in \cite{D1}.
Note in general  that $A(n,k)= 0$ for $k > y_n$ by definition. Thus one must be careful to apply the result only when $k \le y_n$.

We illustrate with a new example, where the input sequence is the Fibonacci sequence. Here, as is standard, 
$$ F_1 =1, F_2=1, F_3 =2, F_4 = 3, F_5= 5, F_6 = 8, \cdots $$

The first eight rows and nine columns of the output array turn out to be
$$ \begin{pmatrix} \label{fibonanni-output}  1 & 1 & 0 & 0 & 0 & 0 & 0 & 0 & 0 \cr 1 & 2 & 0 & 0 & 0 & 0 & 0 & 0 & 0 \cr 1 & 3 & 3 & 0 & 0 & 0 & 0 & 0 & 0 \cr 1 & 4 & 7 & 7 & 0 & 0 & 0 & 0 & 0 \cr 1 & 5 & 12 & 19 & 19 & 19 & 0 & 0 & 0 \cr 1 & 6 & 18 & 37 & 56 & 75 & 75 & 75 & 75 \cr 1 & 7 & 25 & 62 & 118 & 193 & 268 & 343 & 418 \cr
1 & 8 & 33 & 95 & 213 & 406 & 674 & 1017 & 1435 \end{pmatrix}.$$
Note that the $7$-th row has five more copies of $418$, and the $8$-th row includes also $1853, 2271, 2689, 3107, 3525$ with this last number being repeated nine times.
The multiplicities arise because $6=13-8+1$ and $9=21-13+1$.

We find several of the smaller values. First of all $A(n,0)=1$ for all $n$, because there is only one valid $n$-tuple with first entry $0$. Next, $A(n,1)= n$ for all $n$, because
we are counting $n$-tuples of the form $(1,x_2,\ldots, x_n)$ where each $x_j$ is either $0$ or $1$, and the tuple is non-increasing. There are $n$ possibilities, as the total number of ones can be anything from $1$ up to $n$. We can of course compute the rest of the array by using the recurrence, but it is instructive to compute an entry directly. For example, we find $A(5,5)$.
We are considering $5$-tuples whose first entry is $5$. The second entry is the minimum of $3$ and the first entry, and so on. We list all the possibilities (where we do not include the commas or parentheses).
$$ 53211, 53210, 53200, 53111, 53110, 53100, 53000,$$
$$ 52221, 52210, 52200, 52111, 52110, 52100, 52000,$$
$$ 51111,51110, 51100, 51000, 50000$$
Hence there are $19$ possibilities. Notice that $A(5,4)$ and $A(5,3)$ also equal $19$.

\section{Polynomials and explicit formulas}

One might hope to use the recurrence (\ref{recurrence-above}) to determine formulas for $A(n,k)$ in terms of the first row. Doing so is not possible, because $A(n,k)=0$ for $k > y_n$.
If we regard the entries $A(1,j)$ as parameters, then they are either $1$ or $0$, depending on only $y_1$ in the input sequence. 
In order to take the full sequence into account, we must remember that the recurrence applies only when $k$ is sufficiently small, namely $k \le y_n$. 
We will therefore write out formulas for $A(N+n-1,k)$ when $N$ is sufficiently large and study the resulting sub-array. 

The first row of ${\mathcal A}$ is always
$$ \begin{pmatrix} 1 & A(1,1) & A(1,2) & \cdots & A(1,k) & \cdots \end{pmatrix}. $$
By (\ref{recurrence-above}), for $k$ sufficiently small,  the second row is 
$$ \begin{pmatrix} 1 & 1+ A(1,1) & 1+ A(1,1) + A(1,2) & \cdots & 1+ \sum_{j=1}^k A(1,j) & \cdots  \end{pmatrix} .$$
To anticipate the general statement we also write the third row in this case:
$$ \begin{pmatrix} 1 & 2 + A(1,1) & 3+ 2 A(1,1) + A(1,2) & \cdots & k+1+ \sum_{j=1}^k c_j A(1,j) & \cdots  \end{pmatrix} .$$
Here $c_j = k+1-j$.  For larger index $n$, the $c_j$ are explicit binomial coefficients.

Fix a sufficiently large integer $N$. We wish to do the analogous thing when we use the $N$-th row as the starting point. 
We write out the array $A(N+ n-1,k)$ as a sub-array of ${\mathcal A}$, 
for $1\le n \le 5$ and $0 \le k \le 4$. To avoid too many indices, we put $A(N,1) =1$, $A(N,2)=a$, $A(N,3)=b$, $A(N,4)=c$, and $A(N,4) = d$. By repeated application of
(\ref{recurrence-above}) we obtain

$$ \begin{pmatrix} 1 & a & b & c & d \cr
1 & 1+a & 1+ a + b & 1 + a +b + c & 1+ a +b + c + d \cr
1 & 2+a & 3 + 2a+b & 4+ 3 a + 2b + c &  5+ 4a + 3b +2 c +d \cr
1 &  3 + a & 6 + 3a + b  & 10 + 6a + 3b + c & 15 + 10 a + 6 b + 3c + d \cr
1 & 4 + a & 10 + 4a + b & 20 + 10 a + 4b + c & 35 + 20 a + 10 b + 4c + d
\end{pmatrix}. 
$$
It follows that $A(N+n-1,k)$  equals ${N+n+k-3 \choose k}$ plus a linear combination of the parameters $A(N,j)$ for $1 \le j \le k$. 
These formulas hold of course under the assumption that $N$ is sufficiently large.  
We obtain the following result.

\begin{theorem} \label{array-formula} 
Let $(y_n)$ be an input sequence with output array
$A(n,k)$.  Assume that $N$ is such that
$k \leq y_{N+1}$.  Then, for $n \geq N+1$,

\begin{equation} \label{decisive} 
A(n,k)= {n+k-2 \choose k}+ \sum_{j=1}^k { n+k-2-j \choose k-j} A(N,j) \end{equation} \end{theorem}

\begin{proof} 
The $N$-th row of ${\mathcal A}$ by definition has entries $1,A(N,1), A(N,2), \ldots ,A(N,k)$. 
By formula (\ref{recurrence-above}), the $k$-th entry in the $m$-th row is the sum of the entries in the $(m-1)$-st row
up to $k$. Thus (\ref{decisive}) holds for $n=N+1$. The conclusion then follows by induction and Pascal's formula for sums of binomial coefficients.
\end{proof}

Henceforth we write $N(k)$ for the smallest $N$ such that $k \leq y_{N+1}$. As long as $(y_n)$ is not eventually constant, such an $N(k)$ exists.
The situation where the input sequence $(y_n)$ is eventually constant is a bit annoying. 
In this case $A(n,k)=0$ for $k$ sufficiently large. The next corollary therefore requires the assumption
that $y_n$ tends to infinity. 

\begin{corollary} \label{poly-corollary} Assume that the input sequence $(y_n)$ is not eventually constant. 
Suppose $n$ is large enough such that $k \le y_n$. Then the entry $A(n,k)$ is the value of a polynomial $p_k$ of degree $k$ in $n$. \end{corollary} 
\begin{proof} Since the input sequence tends to infinity, we can always choose an $N$ such that $N+1 \ge n$ and apply the theorem.
The binomial coefficient ${m \choose k}$ is a polynomial of degree $k$ in $m$ and the terms from the linear combination in (\ref{decisive})
are of lower order. A polynomial of degree $k$ in $m$ is also a polynomial of degree $k$ in $n$ when $m = n+k-2$. The result follows. \end{proof}

Note for all $n$ that $ A(n,0)=1$.  Unless $y_1=0$, we have $ A(n,1) = n + A(1,1)-1$. We also generally have
$$ A(n,2) = \frac{n^2 + n} {2} + A(1,2) -1. $$
General formulas for $A(n,k)$ for larger $k$ require assuming that $n\ge N(k)$. 
We therefore provide an alternate derivation of the polynomials from Theorem \ref{array-formula}.

Putting $A(n,k+1) = \lambda^n$ in (\ref{recurrence-left}), we see that the solutions to the homogeneous equation $A(n+1,k+1) - A(n,k+1) = 0$ are constants. 
By induction on $k$ it follows for $n\ge N(k)$ that
$A(n,k+1)$ is a polynomial $p_k(n)$ of degree $k$ in $n$. We compute these polynomials using Lagrange interpolation.

We next discuss these polynomials for several input sequences.
For the bracket array we have the following result from \cite{D1}:

\begin{proposition} \label{poly-bracket} For the bracket array, the polynomials $p_k$ described  in Corollary \ref{poly-corollary} are given for $k \le 6$ by 
\begin{align*}
& p_0(x) = 1 & \\
& p_1(x) = x & \\
& p_2(x) = \frac{x^2}{2} + \frac{x}{2}  -1 & \\
&  p_3(x) = \frac{x^3}{6} + \frac{x^2 }{ 2} - \frac{2x} {3} - 2 & \\
&  p_4(x) = \frac{x^4}{24} + \frac{x^3} {4}- \frac{x^2}{24} - \frac{9x}{4} + 2 & \\
& p_5(x) = \frac {x^5} {120} + \frac{x^4}{12} + \frac{x^3} { 8} - \frac{13x^2}{12} + \frac{13x} {15}  - 5 & \\
&  p_6(x) = \frac{x^6 } {720} + \frac{x^5} {48} + \frac{11 x^4} {144} - \frac{13x^3}{48} - \frac{7x^2} {90} -\frac{19x}{4}  + 10.
\end{align*}
\end{proposition}
We next compute them for the Catalan triangle. 
\begin{proposition} \label{Cat-polynomial} The polynomials defining the $k$-th column of the Catalan triangle (for $n \ge k $) are given by 
\begin{itemize}
\item$p_0(x) = 1$
\item $p_1(x) = x$
\item $p_2(x) = \frac{x^2+x -2}{2} = \frac{(x-1)(x+2)}{2}$
\item $p_3(x) = \frac{x^3 + 3 x^2 -4x -12}{6} = \frac {(x-2) (x+2) (x+3)}{6}$
\item $p_4(x) = \frac{ x^4 + 6x^3 -x^2 - 54 x -72} {24}= \frac{(x-3)(x+2)(x+3)(x+4)} {24}$
\end{itemize}
\end{proposition}
\begin{proof} We omit the details, but provide the idea in the special case $k=3$. We wish to find the unique polynomial  $p$ of degree $3$ that satisfies
$$ p(3)=5, \ p(4)= 14, \  p(5)=28, \ p(6)=48. $$
We could also have used the value $p(2)=0$. 
The InterpolatingPolynomial command in Mathematica does the job for us. \end{proof}

\begin{proposition} \label{Fib-polynomial} The polynomials defining the $k$-th column of the Fibonacci output array  (for $n \ge k $) are given by 
\begin{itemize}
\item $p_0(x) = 1$
\item $p_1(x)=x$
\item $p_2(x) = \frac{(x-2)(x+3)}{2} $
\item $p_3(x) = \frac {(x-3 ) (x^2+6x+2)}{6}$
\item $p_4(x) = \frac{(x-4)(x+1)(x^2+9x+6) }{24}$.

\end{itemize}
\end{proposition}

\begin{remark} Theorem \ref{array-formula} implies that the top order term in $p_k(n)$ equals $\frac{n^k}{k!}$ for every input sequence that is not eventually constant. \end{remark} 

We close this section with a somewhat perverse input sequence that illustrates why we must consider $N(k)$ in solving the recurrence. 

\begin{example} \label{mult-sequence} Consider the input sequence with values
$$ 1,2,2,3,3,3,4,4,4,4,5,5,5,5,5,6,6,6,6,6,6, \ldots $$
Thus $1$ appears once, $2$ appears twice, $3$ appears three times, and so on. Using the code from the next section, one can show for example that $A(n,4)=0$ for $1\le n \le 6$.
The polynomial $p_4$ is given by
$$   \frac{x^4}{24} + \frac{x^3}{4} - \frac{x^2}{24} - \frac{29x}{4} - 63.     $$
Note that $y_7=4$. The values of $p_4(n)$ agree with $A(n,4)$ only when $n\ge 7$. Also $A(n,5) = 0$ for $1 \le n \le 10$. Note that $y_{11}=5$. The polynomial $p_5$ is given by 
$$ \frac{x^5}{120} + \frac{x^4}{12} + \frac{x^3}{8} - \frac{43x^2}{12} - \frac{1999x}{30} - 767. $$
The values here agree with $A(n,5)$ only when $n\ge 11$. These formulas show that the formula $A(n,k) = p_k(n)$ holds only when $n\ge N(k)$. The explanation here is that the input sequence is not strictly increasing.\end{example}

\section{Sage code}\label{sage}

The following code works as follows. In the first line, replace $F(n)$ with the formula for $y_n$ and $R$ with the desired array size, or list as many of the terms as you wish. 
The output provides the input sequence,
the output array up to size $R$, and the output sequence of row sums. The code relies on the recurrence from Theorem \ref{recurrence-theorem},
the initial condition that the first column of the array is all ones, and that the first row has the appropriate number of ones.

\begin{verbatim}

y=[F(n) for n in range(1,R+1)];
show(y);
N=len(y);
K=max(y);
A=matrix.zero(N,K+1);
for n in range(0,N):
 A[n,0]=1;
 for k in range(0,y[0]+1):
A[0,k]=1;
for n in range(1,N):
 for k in range(1,y[n]+1):
A[n,k] = A[n-1,k]+A[n,k-1];
show(A);
show(sum(A.T));  \end{verbatim}

\section{Comparison with Pascal's triangle}

The recurrence relation (\ref{recurrence-left}) looks superficially like Pascal's formula. We recall (\ref{recurrence-left}) here for ease of comparison:
$$ A(n+1,k+1) = A(n+1,k) + A(n,k+1). \leqno (8) $$

By contrast, putting ${n \choose k} = A(N,k)$, Pascal's formula says 
\begin{equation}  \label{pascal} A(n+1,k+1) = A(n,k) + A(n,k+1). \end{equation}

Nonetheless Pascal's triangle fits into our discussion by way of the following approach. We create an array using the recurrence (\ref{recurrence-left}) by setting all the entries in 
both the far left columns {\it and} the top row equal to $1$.  We call this array the maximal output array $A_M$. Thus
we set $A_M(n,0) = 1$ for all $n$ and $A_M(1,k) = 1$ for all $k$. The recurrence (\ref{recurrence-left}) then determines the full array. 
One way to think about the maximal array is to imagine that the input sequence satisfies
$y_n=\infty$ for all $n$. The following theorem justifies this notion by providing estimates for arbitrary output arrays.

\begin{theorem} \label{Pascal-ineq}
Given a nondecreasing sequence of nonnegative numbers $(y_n)$, let
$A(n,k)$ be its output array and $W(n)$ its output sequence, namely the
row sums of $A(n,k)$.  Then, if $n\ge 2$ and $k \le y_n$ for the lower bound in (\ref{lebl-estimate}), and for all $n,k$ in the upper bounds, 
\begin{equation} \label{lebl-estimate}
 \min(y_{n-1}+1 , k+1 )\le A(n,k) \leq \binom{n-1+k}{k} \end{equation}
 \begin{equation} \label{lebl-row} W(n) \leq \binom{n+y_n}{y_n} .
\end{equation}
\end{theorem}

\begin{proof}
The first inequality in (\ref{lebl-estimate}) is immediate from the definition of $A(n,k)$.  To prove the second inequality, 
we start by setting $A_M(n,0) = 1$ for all $n$ and $A_M(1,k) = 1$ for all
$k$.  
We then use the usual recurrence relation
\begin{equation}
A_M(n+1,k+1) = A_M(n+1,k)+A_M(n,k+1)
\end{equation}
to fill in the rest of the entries in the output array. The result is a rotated version of the
familiar Pascal's triangle:
\begin{equation}
\begin{pmatrix}
1 & 1 & 1 & 1 & \dots \\
1 & 2 & 3 & 4 & \dots \\
1 & 3 & 6 & 10 & \dots \\
1 & 4 & 10 & 20 & \dots \\
\vdots & \vdots & \vdots & \vdots & \ddots
\end{pmatrix}.
\end{equation}
Thus
\begin{equation}
A_M(n,k) = \binom{n-1+k}{k} .
\end{equation}

Given the input sequence $( y_n )$, we construct its output array
$A(n,k)$.  We do so as usual. We put 
$A(n,0) = 1$ for all $n$ and $A(1,k) = 1$ for $k \leq y_1$.  We then use
the recurrence
$$A(n+1,k+1) = A(n+1,k)+A(n,k+1)$$
{\bf only} for those elements where $k+1 \leq y_{n+1}$,
while setting the rest of the elements to zero.  It follows that
\begin{equation}
A(n,k) \leq A_M(n,k) = \binom{n-1+k}{k} .
\end{equation}
Let $W(n)$ be the $n$-th row sum of ${\mathbb A}$, that is, the output
sequence.  We also obtain an estimate on $W(n)$,
as each row sum appears as an entry in the next row of the output array. In particular, since $(y_n)$ is
nondecreasing, then
\begin{equation}
W(n) = \sum_{k=0}^{y_n} A(n,k) = A(n+1,y_n) \leq
A_M(n+1,y_n) = \binom{n+y_n}{y_n} .
\end{equation}
\end{proof}

\begin{corollary} The map $\Phi:{\mathcal S} \to {\mathcal S}$ taking input sequences to output sequences is not surjective. \end{corollary}

These inequalities are far from sharp, and become less sharp the faster the input sequence grows. In the next section we provide improved bounds.

\section{Three part harmony}

In order to better understand the output array, it is useful to consider three types of terms in the $n$-th row. We make the following definitions.

\begin{definition} \label{three-part} Let $(y_n)$ be an input sequence and let $A(n,k)$ denote its output array. For each $n$ with $n\ge 3$ we put
$$ W(n) = \sum_{k=0}^\infty A(n,k) = \sum_{k=0}^{y_n} A(n,k) $$

$$  T(n) = \frac {(1+ y_n - y_{n-1}) W(n-1)} {W(n)} $$

$$ M(n) = \frac{ \sum_{k=1}^{(y_{n-1}- y_{n-2}) + 1}  (W(n-1)- kW(n-2)) } {W(n)}$$

$$ S(n) = 1 - M(n) - T(n).$$ \end{definition}

The idea is easy to explain. Of course $W(n)$ is simply the $n$-th row sum. There are three kinds of terms contributing to this sum. For a typical output array,
the largest term in the $n$-th row is $W(n-1)$, and this term gets repeated for each $k$ in the interval $y_{n-1} \le k \le y_n$. Thus there are $1+y_{n}-y_{n-1}$ such terms.
Hence $T(n)$ denotes the fraction of terms of this type in the $n$-th row.
Because of the recurrence relation determining $A(n,k)$, we also have terms of the form $W(n-1) - j W(n-2)$. There are 
$$R= (y_{n-1}- y_{n-2}) + 1$$
 such terms. ($R$ is the number of integers $j$ with $y_{n-2} \le j \le y_{n-1}$.)
Thus $M(n)$ denotes the fraction of terms of this type in the $n$-th row. Finally $S(n)$ denotes the fraction of the remaining terms.

\begin{remark} The letter $T$ stands for {\it top}, the letter $M$ stands for {\it middle}, and the letter $S$ stands for {\it small}. \end{remark}

This method of organizing the terms leads to bounds from both sides for $W(n)$.

\begin{theorem} Let $(y_n)_{n \ge 1}$ be an arbitrary input sequence. Put $y_0 =0$ for convenience. Then
\begin{equation} \label{improved-row}  \prod_{k=0}^{n-1} (1+y_{k+1} - y_k) \le W(n) \le  \prod_{k=1}^n (1+y_k) \end{equation} 
\end{theorem}
\begin{proof} We begin with the left-hand inequality. Consider the $n$-th row of ${\mathcal A}$. Since the number $W(n-1)$ gets repeated $(1+y_n - y_{n-1})$ times in this row,
we must have
\begin{equation} W(n) \ge (1+y_n - y_{n-1}) W(n-1). \end{equation} 
It then follows inductively that
\begin{equation} W(n) \ge \prod_{k=1}^{n-1} (1+y_{k+1} - y_{k}) W(1). \end{equation} 
Since $W(1)$ is the number of valid $1$-tuples, $W(1) = 1 + y_1$. Using the convention that $y_0=0$, we can write $W(1)$ as $1+ y_1 - y_0$, and the inequality follows.

To prove the upper bound, we observe that the maximum entry in the $n$-th row of ${\mathcal A}$ is $W(n-1)$.  Hence $W(n)$ is at most the number of non-zero terms in the row times
$W(n-1)$. For non-zero entries, the column index $k$ varies from $0$ to $y_n$, and hence (using the convention that $W(0)=1$), 
\begin{equation} W(n) \le (1+y_n) W(n-1) \le (1+y_n)(1+y_{n-1} )W(n-2) \le \ldots \le \prod_{k=1}^n (1+y_k) \end{equation} \end{proof}

We now compute the numbers $T(n), M(n),S(n)$ for the Catalan triangle. In this case, as is well known, $W(n)$ is the $(n+1)$-st Catalan number. 

\begin{theorem} Put $y_n=n$. The output array generated by this sequence is the Catalan triangle. For this array  we have the following results:
$$ W(n) = C_{n+1} = \frac{1} {n+2} {2n+2 \choose n+1},$$
$$ T(n) = \frac{n+2}{2n+1},$$
$$ M(n) = \frac{(n+2)(5n-7)}{4 (2n+1) (2n-1)},$$
$$ S(n) = \frac{3(n-3)(n-2)}{4(2n+1)(2n-1)}. $$
Therefore we have the following asymptotic results:
\begin{itemize}
\item $\lim_{n \to \infty} T(n) = \frac{1} {2},$
\item $\lim_{n \to \infty} M(n) = \frac{5}{16},$
\item $\lim_{n \to \infty} S(n) = \frac{3}{16}. $
\end{itemize}\end{theorem}
\begin{proof} Since the limits are immediate from the stated formulas, it suffices to prove the formulas.
We assume the known result that $W(n)$ is the Catalan number $C_{n+1}$, which equals  $ \frac{1} {n+2} {2n+2 \choose n+1}$. This number gets repeated twice in the $(n+1)$-th row of  the Catalan triangle,
since
$$(y_{n+1} - y_n )+ 1 = (n+1-n)+1=2. $$
Therefore
$$ T(n) = {2 \frac{1} {n+1} {2n \choose n} \over \frac{1} {n+2} {2n+2 \choose n+1}}= {2 (n+2) \over n+1} \ {(2n)! (n+1)! (n+1)! \over n! n! (2n+2)!} = \frac{n+2}{2n+1}.$$
To compute $M(n)$, we note that the sum in the numerator of its definition also has two terms. We thus we get two copies of $W(n-1)$ and $3=1+2$ copies of $W(n-2)$ when computing
$M(n)$. Thus 
$$ M(n) = \frac{2 C(n-1) - 3 C(n-2)}{C(n)}, $$
which simplifies to 
$$ M(n) = \frac{(n+2)(5n-7)}{4 (2n+1) (2n-1)}, $$
as claimed. Since $T(n)+ M(n)+S(n)=1$ by definition, the formula for $S(n)$ follows as well.
\end{proof}

\begin{proposition} Let $y_n$ be a constant sequence. Then $\lim_{n\to \infty} T(n) = 1$, and hence $\lim_{n \to \infty} M(n) = \lim_{n\to \infty} S(n) = 0$. \end{proposition}
\begin{proof} Suppose $y_n = j$ for all $n$. One can check, in analogy with Example \ref{constant5},  that $W(n) = {n+j \choose j}$ and hence
$$ \frac{W(n-1)}{W(n)} = \frac{ (n-1+j)!n! j!} {j! (n-1)! (n+j)!} = \frac{n} {n+j}.$$
In the $n$-th row, there is only one term of the form $W(n-1)$. Hence this ratio yields $T(n)$. It is evident that $\lim_{n\to \infty} T(n) = 1$. 
\end{proof}

\begin{example} For the bracket tournament output array, the asymptotic limiting values of $T(n)$, $M(n)$, and $S(n)$ satisfy the following.
They agree to many decimal places with the values when $n=82$.  Using exact formulas for $W(80)$, $W(81)$, and $W(82)$, we obtain the following 
$$ T(82) \approx 0. 744039272799855,$$ 
$$ M(82) \approx 0.233621026532793, $$
$$ S(82) \approx 0. 022339700667352. $$
 \end{example}
 
 \begin{remark} For the tournament bracket output array, Alois Heinz has found the following formula, expressing $W(n)$ in terms of previous values:
 \begin{equation}  \label{alois} W(n) = \sum_{j=0}^{n-1} W(j) (-1)^{n-j+1} {1+2^j \choose n-j}. \end{equation}
 This formula presumes that ${a \choose b} = 0$ when $b > a$ and also that $W(0)=1$.
 See sequence A355519 from \cite{OEIS}. The summand when $j=n-1$ in formula (\ref{alois}) is precisely what we used to define the numerator of $T(n)$.
 \end{remark}
 
 \begin{remark} Consider the output array for the Fibonacci sequence.  For $n=24$, one can obtain the approximate values
 $$ T(24) \approx .678$$
 $$ M(24) \approx .277$$
 $$ S(24) \approx .044.$$
 These numbers lie between the corresponding values for the Catalan output array and those of the bracket output array.
 The reason is that the growth of the Fibonacci numbers lies between the growth of the input sequences $n$ and $2^{n-1}$. \end{remark}
   
 \section{Examples of input and output sequences}

 We list some input sequences $(y_n)$ and the first few terms of (or a formula for) the output sequence $W(n)$ of row sums. We then note whether the output sequence appears in \cite{OEIS}.
 
\begin{example} We provide many input sequences, the corresponding output sequence, and whether or not the output sequence appears in \cite{OEIS}.
 \begin{enumerate}
 \item $y_n=n$. Then $W(n) = \frac{1}{n+2} {2n+2 \choose n+1}$. Sequence A00108 in \cite{OEIS}. (Catalan numbers.)
 \item $y_n = \frac{1}{n+1} {2n \choose n}$. Then $W(n) = 2,5,14,287, \dots$. Sequence NOT in \cite{OEIS}.
 \item $y_n = 2n-1$. Then $W(n) = \frac{1}{n+1}{3n+1 \choose n}$. Sequence A006013 in \cite{OEIS}.
 \item $y_n = 2n$. Then $W(n) = \frac{1} {2n+3} {3n+3 \choose n+1}$.  Sequence A001764 in \cite{OEIS}.
 \item $y_n= 3n$. Then $W(n) = \frac{1}{3n+4} {4n+4 \choose n+1} $. Sequence A002293 in \cite{OEIS}.
 \item $y_n =4n $. Then $W(n) = \frac{1}{4n+5}{5n+5 \choose n+1}$. Sequence A002294 in \cite{OEIS}.
 \item $ y_n = {n+1 \choose 2}$. Then $W(n) = 2,7,37, 268, 2496, \ldots$. Sequence  A107877 in \cite{OEIS}.
 \item $y_n = n^2$.  $W(n) = 2,9, 70,805, \ldots$. Sequence  A177450 in \cite{OEIS}.
 \item $y_n = n^2 + n$. $W(n) = 3,18,172, 2313, 40626, \dots.$ Sequence  A177447 in \cite{OEIS}.
 \item $y_n = n^3$. $W(n) = 2,17,404,20002, \ldots$. Sequence NOT in \cite{OEIS}.
 \item $y_n = (1,1,2,3,5,8,13,21, \ldots)$ is the Fibonacci sequence. Then $W(n) = 1,2,3,7,19,75,418, \ldots$. Sequence NOT in \cite{OEIS}.
 \item $y_n = 2^{n-1}$. Then $W(n) = 2, 5,19, 123, 1457, \ldots$. Sequence A355519 in \cite{OEIS}.
 \item $y_n = 3^{n-1}$. Then $W(n) = 2,7,58,1317, \ldots$. Sequence NOT in \cite{OEIS}.
 \item $y_n = (2,3,5,7,11,13,17,19, 23, 29, 31, 37, 41, 43, 47)$. (the first fifteen primes.)  Then $W(n)= 3,9, 37, 173, 1217, 7557, 60803, 419255, \ldots$. This sequence is not in \cite{OEIS}.
 \item $y_n = \lfloor{ (1.5)^n}\rfloor$. Then $W(n)=2,5,14,56,258,1803, 18352, \ldots$. This sequence is not in \cite{OEIS}.
 \item $y_n$ is the sequence defined in Example \ref{mult-sequence}. Then $W(n) = 2,5,9,23,43,70, \dots$. This sequence is not in \cite{OEIS}.
 \end{enumerate}

\end{example}

\begin{remark} It is natural to wonder whether studying the output sequence from the input sequence of primes could have any value.
The first author suspects otherwise, simply because the output sequence grows so rapidly. On the other hand, perhaps studying the numbers from Definition \ref{three-part} for this sequence could be interesting. \end{remark}

 \begin{remark} The output sequence in item (12) was entered by the first author into \cite{OEIS}
 in July, 2022. It seems that this sequence should be better known, as it counts something (tournament brackets) that is very natural. \end{remark}
 
 \begin{remark} The sequence $\lfloor{ (1.5)^n}\rfloor$ from item (15) arises in various places. It is sequence A002379 in \cite{OEIS}. This sequence could be useful in the context of this paper because
 of its similarity to the Fibonacci sequence.  \end{remark}

 \begin{definition} Recall that $\Phi$ denotes the operation of assigning an output sequence to an input sequence.
 The examples above show that $\Phi$ is not linear.  Note that
 item (2) above  can be regarded as applying $\Phi$ twice. It seems difficult to determine the range of $\Phi$, other than the inequalities from Theorem \ref{Pascal-ineq}.  \end{definition}

 We conclude with what seems to be a difficult problem.

 {\bf Open problem}: Compute the growth rate of the output sequence in terms of the growth rate of the input sequence. 
 
 \section{Acknowledgements}
 
 The authors wish to acknowledge useful discussions with Dusty Grundmeier. In particular Grundmeier has rewritten the Sage code in Mathematica, 
 which was more convenient for the first author.
 The first author especially wishes to acknowledge Alois Heinz, whom he has met only over e-mail, for his contributions to sequence A355519 in \cite{OEIS}.
 Heinz also wrote that he was looking forward to additional new sequences, and this paper provides (infinitely many) such examples. 
 Thus his encouraging e-mail provided the impetus for this paper.

\end{document}